\documentclass[11]{amsart}

\usepackage{amsfonts}
\usepackage{amscd}
\usepackage{amsmath, mathrsfs, amssymb}
\usepackage{amsthm}
\usepackage{setspace}
\usepackage{hyperref}
\usepackage{color}
\usepackage{epsfig}
\usepackage{here}
\usepackage{graphicx}
\usepackage[all]{xy}
\usepackage{psfrag}
\usepackage{graphicx,transparent}
\usepackage{enumerate}
\usepackage{caption}
\usepackage[style=alphabetic, isbn=false, doi=false, url=false, backend=bibtex, maxnames=5, maxalphanames=5, giveninits=true]{biblatex}

\theoremstyle{plain}
\newtheorem{theorem}{Theorem}[section]

\newtheorem{corollary}[theorem]{Corollary}
\newtheorem{claim}[theorem]{Claim}

\theoremstyle{definition}

\newcommand{\PP}{\mathcal P}

\newcommand{\BPP}{\overline{\mathcal P}}
\newcommand{\BNN}{\overline{\mathcal N}}

\newcommand{\calN}{\mathcal N}

\newcommand{\GL}{\operatorname{GL}}

\newcommand{\bbC}{\mathbb C}

\newcommand{\bbR}{\mathbb R}

\newcommand{\MS}{\mathcal{MS}}

\newcommand{\TPP}{\widetilde{\mathcal P}}
\newcommand{\TNN}{\widetilde{\mathcal N}}

\title[]{Connectedness of the boundaries of the strata of differentials}

\begin{document}

\date{\today}

\author{Dawei Chen}

\address{Department of Mathematics, Boston College, Chestnut Hill, MA 02467, USA}
\email{dawei.chen@bc.edu}

\begin{abstract}
Let $\PP(\mu)^{\circ}$ be a connected component of the projectivized stratum of differentials on smooth complex curves, where the zero and pole orders of the differentials are specified by $\mu$. When the complex dimension of $\PP(\mu)^{\circ}$ is at least two, Dozier--Grushevsky--Lee, through explicit degeneration techniques, showed that the boundary of $\PP(\mu)^{\circ}$ is connected in the multi-scale compactification constructed by Bainbridge--Chen--Gendron--Grushevsky--M\"oller. 

A natural question is whether the connectedness of the boundary of $\PP(\mu)^{\circ}$ is determined by its intrinsic properties. In the case of meromorphic differentials, we provide a concise explanation that the boundary of $\PP(\mu)^{\circ}$ is always connected in any complete algebraic compactification, based on the fact that the strata of meromorphic differentials are affine varieties. We also observe that the same result holds for linear subvarieties of meromorphic differentials, as well as for the strata of $k$-differentials with a pole of order at least $k$. 

In the case of holomorphic differentials, using properties of Teichm\"uller curves, we provide an alternative argument showing that the horizontal boundary of $\PP(\mu)^{\circ}$ and every irreducible component of its vertical boundary intersect non-trivially in the multi-scale compactification. 
\end{abstract}

\maketitle
     

\section{Introduction and results}
\label{sec:intro}

For $\mu = (m_1, \ldots, m_n)$ with $\sum_{i=1}^n m_i  =  2g-2$, let $\PP(\mu)$ denote the projectivized stratum of differentials on smooth, connected complex curves of genus $g$, whose orders of zeroes and poles are specified by the signature $\mu$. Equivalently, $\PP(\mu)$ parameterizes (possibly meromorphic) canonical divisors of type $\mu$.  

The study of differentials is important in surface dynamics and moduli theory; see \cite{ZorichSurvey, WrightSurvey, ChenSurvey} for introductions to this fascinating subject. Despite various advances, the topology of $\PP(\mu)$ remains mysterious in general. Among the few known results, Kontsevich--Zorich and Boissy classified the connected components of $\PP(\mu)$ in the cases of holomorphic and meromorphic differentials, respectively; see \cite{KontsevichZorich, BoissyConnected}. We use $\PP(\mu)^{\circ}$ to denote a connected component of $\PP(\mu)$, where $\PP(\mu)^{\circ} = \PP(\mu)$ when the latter is irreducible. 

A powerful method for studying an open moduli space is through the boundary of its compactification. We use $\MS(\mu)$ to denote the moduli space of multi-scale differentials, constructed in \cite{BCGGM1, BCGGM3}, which compactifies $\PP(\mu)$ and is a complex orbifold with normal crossings boundary. Since its inception, the multi-scale compactification has served as the foundation for numerous further developments; see \cite{CMSZVolumes, CMZEuler,  BDGLinear, BenirschkeLinear, CCMKodaira}. For a connected component $\PP(\mu)^{\circ}$ of $\PP(\mu)$, we use $\MS(\mu)^{\circ}$ to denote the corresponding component in the multi-scale compactification.  

In \cite{BoissyEnds}, Boissy showed that for all holomorphic signatures $\mu$, every connected component $\PP(\mu)^{\circ}$ of $\PP(\mu)$ has only one topological end. The proof uses ideas from the geometry and dynamics of flat surfaces, including the Veech zippered rectangles construction and the corresponding Rauzy classes. In the recent work \cite{DGLEnds}, Dozier--Grushevsky--Lee recovered Boissy's result and generalized it to the case of meromorphic differentials. Using explicit degeneration techniques, they showed that the boundary of the multi-scale compactification $\MS(\mu)^{\circ}$ is connected whenever $\dim_{\bbC}\MS(\mu)^{\circ}\geq 2$, which immediately implies that $\PP(\mu)^{\circ}$ has only one topological end; see \cite[Lemma 3]{DGLEnds}.  

A natural question is whether the connectedness of the boundary of $\PP(\mu)^{\circ}$ is determined by its intrinsic properties. The purpose of this article is to answer this question in the case of meromorphic differentials. 

Before stating our results, we remark that there are two types of boundary divisors in the multi-scale compactification. One is called the horizontal boundary, which parameterizes multi-scale differentials with at least one horizontal node---meaning that on both branches of the node, the differential has simple poles with opposite residues. The other is called the vertical boundary, which parameterizes multi-scale differentials whose level graphs have at least two levels. 

We denote by $\Delta_{\rm H}$ and $\Delta_{\rm V}$ the horizontal and vertical boundary divisors of $\MS(\mu)^{\circ}$, respectively; see \cite[Section 2.3]{BCGGM3} for their precise definitions. In general, $\Delta_{\rm H}$ and $\Delta_{\rm V}$ can be reducible, and their components may intersect in various ways. We denote by $\Delta'_{\rm H} = \Delta_{\rm H}\setminus (\Delta_{\rm H}\cap \Delta_{\rm V})$ the locus of stable differentials with horizontal nodes only, i.e., those whose level graphs consist of a single level and contain at least one horizontal edge.  

\begin{theorem}
\label{thm:main}
Let $\mu$ be a meromorphic signature such that $\dim_{\bbC} \PP(\mu)^{\circ}\geq 2$. 
\begin{itemize}
\item[(1)] Let ${\BPP(\mu)^{\circ}}$ be an irreducible, complete variety containing $\PP(\mu)^{\circ}$ as an open subset. 
Then the complement of $\PP(\mu)^{\circ}$ in ${\BPP(\mu)^{\circ}}$ is connected. 
\item[(2)]  Let ${\TPP(\mu)^{\circ}}$ be an irreducible, complete variety containing $\PP(\mu)^{\circ}\sqcup \Delta'_{\rm H}$ as an open subset. Then the complement of $\PP(\mu)^{\circ}\sqcup \Delta'_{\rm H}$ in $\TPP(\mu)^{\circ}$ is connected. 
\end{itemize}
\end{theorem}

\begin{proof}
In \cite[Theorem 1.2]{ChenNonvarying}, for meromorphic signatures $\mu$, 
it was shown that $ \PP(\mu)^{\circ}$ is an affine variety, which implies (1) by applying \cite[p. 166, Corollary of Theorem 1]{GoodmanAffine}.  

Additionally, the proof in \cite[Section 4.4]{ChenNonvarying} applies verbatim to $\PP(\mu)^{\circ}\sqcup \Delta'_{\rm H}$, because the divisor class relations therein still hold over the locus of stable differentials with horizontal nodes only.  Therefore, $\PP(\mu)^{\circ}\sqcup \Delta'_{\rm H}$ is also an affine variety, which implies (2) as well.  
\end{proof}

Applying Theorem~\ref{thm:main} to the multi-scale compactification immediately yields the following result. 

\begin{corollary}
\label{cor:multi}
Let $\mu$ be a meromorphic signature such that $\dim_{\bbC} \PP(\mu)^{\circ}\geq 2$. Then both the total boundary $\Delta_{\rm H}\cup \Delta_{\rm V}$ and the vertical boundary $\Delta_{\rm V}$ are connected.  
\end{corollary}

Note that any closed subset of an affine variety is affine; see, e.g., \cite[p. 106, Theorem 3]{MumfordRed}. Therefore, the above results also hold for subvarieties in the strata of meromorphic differentials. 

\begin{corollary}
\label{cor:sub}
Let $\mu$ be a meromorphic signature, and let $\calN$ be an irreducible subvariety in $\PP(\mu)^{\circ}$ such that $\dim_{\bbC}\calN \geq 2$. 
\begin{itemize}
\item[(1)] Let $\BNN$ be any irreducible, complete variety containing $\calN$ as an open subset. Then the complement of $\calN$ in ${\BNN}$ is connected. 
\item[(2)]  Let $\TNN$ be any irreducible, complete variety containing $\calN\sqcup \Delta'_{\rm H}(\calN)$ as an open subset, where $\Delta'_{\rm H}(\calN)$ is the locus of stable differentials with horizontal nodes only, contained in the multi-scale boundary of $\calN$. Then the complement of $\calN\sqcup \Delta'_{\rm H}(\calN)$ in $\TNN$ is connected. 
\end{itemize}
\end{corollary}

Prominent examples of subvarieties in $\PP(\mu)^{\circ}$ include linear subvarieties locally defined by linear equations of period coordinates, as well as the strata of $k$-differentials which can be lifted into the strata of one-forms through the canonical covering construction; see \cite{LanneauConnected, ChenMoellerQuadratic, BCGGM2, ChenGendronKdiff} for  introductions to these subjects and related results. Note that a $k$-differential lifts to a meromorphic one-form under the canonical covering construction if and only if the $k$-differential has a pole of order at least $k$. Therefore, Corollary~\ref{cor:sub} implies the following special cases. 

\begin{corollary}
\label{cor:linear-k}
Let $\calN$ be an irreducible linear subvariety in a stratum of meromorphic differentials or an irreducible component of a stratum of $k$-differentials with a pole of order at least $k$. Suppose that $\dim_{\bbC}\calN \geq 2$. Then the total boundary and the vertical boundary of $\calN$ in the multi-scale compactification are both connected. 
\end{corollary}

Since affineness is preserved under finite morphisms, our results hold for all versions of the strata of meromorphic differentials---whether the zeros and poles are fully labeled, partly labeled, or unlabeled.  

Finally, we remark that it is an open question whether all strata of holomorphic differentials are affine varieties; see \cite[Theorem 1.1]{ChenNonvarying} for the known cases in low genus.  Therefore, any explicit boundary description of the strata of holomorphic differentials, such as those in \cite{DGLEnds}, remains precious. To this end, we provide an alternative argument for some key steps of the proof in \cite{DGLEnds} that shows the boundary of any component of the strata of holomorphic differentials is connected in the multi-scale compactification. 

\begin{claim}[{\cite[Claims A and C]{DGLEnds}}]
Let $\mu$ be a holomorphic signature. Then the following claims hold: 
\begin{itemize}
\item[(1)] The horizontal boundary $\Delta_{\rm H}$ of $\PP(\mu)^{\circ}$ is non-empty. 
\item[(2)] Every irreducible component of the vertical boundary $\Delta_{\rm V}$ of $\PP(\mu)^{\circ}$ intersects non-trivially with $\Delta_{\rm H}$. 
\end{itemize}
\end{claim}

\begin{proof}
By scaling the differentials to be arbitrarily large, $\PP(\mu)^{\circ}$ contains differentials whose (unprojectivized) period coordinates have integral real and imaginary parts. These differentials correspond to square-tiled surfaces---that is, branched covers of the square torus with a unique branch point; see \cite[Lemma 3.1]{EskinOkounkovVolumes}. Square-tiled surfaces generate (arithmetic) Teichm\"uller curves under the $\GL^{+}_2(\bbR)$-action on $\PP(\mu)^{\circ}$. Teichm\"uller curves are never complete, and their closures intersect only the horizontal boundary; see \cite[Section 3.3]{ChenMoellerAbelian}. This justifies the first claim.  

Let $\Delta_{\Gamma}^{\circ}$ be an irreducible vertical boundary divisor of $\PP(\mu)^{\circ}$, where $\Gamma$ denotes a two-level graph. Note that the top-level stratum of $\Gamma$ parameterizes holomorphic differentials. Therefore, by the above claim, differentials parameterized in the top-level stratum can degenerate to produce horizontal edges. This implies that $\Delta_{\Gamma}^{\circ}$ and $\Delta_{\rm H}$ intersect non-trivially. 
\end{proof}

Now, the connectedness of the total boundary of $\PP(\mu)^{\circ}$ follows from the connectedness of the horizontal boundary $\Delta_{\rm H}$, as noted in \cite[Claim B]{DGLEnds}; the latter, in turn, follows from the classification of the connected components of the strata of meromorphic differentials with two simple poles, given in \cite{BoissyConnected}. 

\section{Acknowledgements} 

This research was supported in part by National Science Foundation Grant DMS-2301030, Simons Travel Support for Mathematicians, and a Simons Fellowship. The author thanks Matteo Costantini, Quentin Gendron, Samuel Grushevsky, Myeongjae Lee, and Martin M\"oller for inspiring discussions on related topics. 

\printbibliography 

\end{document}